\documentclass[11pt,oneside,leqno]{amsart}
\allowdisplaybreaks
\usepackage[margin=2.5cm]{geometry}
\usepackage{multirow,tabularx}
\newcolumntype{Y}{>{\centering\arraybackslash}X}

\usepackage{pbox}
\usepackage{parskip}
\usepackage{eucal}
\usepackage{amscd}
\usepackage{amssymb}
\usepackage{amsmath}
\usepackage{amsfonts}
\usepackage{amsopn}
\usepackage{latexsym}
\usepackage{tikz}
\usepackage{comment}
\usepackage{ulem}
\usepackage{hyperref}
\usepackage[nodisplayskipstretch]{setspace}
\usepackage{caption}
\usepackage[OT2,T1]{fontenc}
\DeclareSymbolFont{cyrletters}{OT2}{wncyr}{m}{n}
\DeclareMathSymbol{\Sha}{\mathalpha}{cyrletters}{"58}
\newcommand{\genlegendre}[4]{%
  \genfrac{(}{)}{}{#1}{#3}{#4}%
  \if\relax\detokenize{#2}\relax\else_{\!#2}\fi
}
\newcommand{\legendre}[3][]{\genlegendre{}{#1}{#2}{#3}}

\usepackage{enumerate}
\usepackage[utf8]{inputenc}

\makeatletter
\@namedef{subjclassname@2010}{%
  \textup{2010} Mathematics Subject Classification}
\makeatother

\newtheorem{thm}{Theorem}[section]
\newtheorem{cor}[thm]{Corollary}
\newtheorem{lem}[thm]{Lemma}

\makeatletter
\newcommand{\BIG}{\bBigg@{2}}
\newcommand{\vast}{\bBigg@{3}}
\newcommand{\Vast}{\bBigg@{5}}
\makeatother

\numberwithin{equation}{section}

\begin{document}
\setlength{\arrayrulewidth}{0.1mm}



\title[Heron triangle and an Elliptic curves]{Construction of an infinite family of elliptic curves of $2$-Selmer rank $1$ from Heron triangles}

\author[Debopam Chakraborty]{Debopam Chakraborty}
\address{Department of Mathematics\\ BITS-Pilani, Hyderabad campus\\
Hyderabad, INDIA}
\email{debopam@hyderabad.bits-pilani.ac.in}

\author[Vinodkumar Ghale]{Vinodkumar Ghale}
\address{Department of Mathematics\\ BITS-Pilani, Hyderabad campus\\
Hyderabad, INDIA}
\email{p20180465@hyderabad.bits-pilani.ac.in}

\author[Anupam Saikia]{Anupam Saikia}
\address{Department of Mathematics and Computation\\ IIT Guwahati\\
Guwahati, INDIA}
\email{a.saikia@iitg.ac.in}

\date{}


\maketitle

\section*{Abstract}

\noindent Given any positive integer $n$, it is well known that there always exist triangles with rational sides $a,b$ and $c$ such that the area of the triangle is $n$. Assuming finiteness of the Shafarevich-Tate group, we first construct a family of infinitely many Heronian elliptic curves of rank exactly $1$ from Heron triangles of a certain type. We also explicitly produce a separate family of infinitely many Heronian elliptic curves with $2$-Selmer rank lying between $1$ and $3$.

\section{Introduction}
\noindent The congruent number problem is to determine which positive integers $n$ appear as the area of a right triangle with rational sides. The problem boils down to identifying the elliptic curve $E_{n}: y^{2} = x^{3} - n^{2}x$, known as the \textit{congruent number elliptic curve}, with positive rank. Currently, there is no algorithm to determine whether a given positive integer $n$ is congruent or not. The Birch and Swinnerton-Dyer conjecture \cite{Birch} predicts that $n$ should be a congruent number if $n \equiv 5$, $6$ or $7 \pmod 8$. Understanding the $2$-Selmer rank plays an important role in the rank computation of an elliptic curve. In \cite{Brown1} and \cite{Brown2}, D. R. Heath-Brown examined the size of the $2$-Semer group of the congruent elliptic curve $E_{n}$. 
A glimpse of the extensive studies on the congruent number problem can be found in \cite{Chahal}, \cite{Coates}, \cite{conrad} and \cite{Johnstone}.

\noindent An immediate generalization of the congruent number problem is the existence of positive integer $n$ as the area of triangles with rational sides without the constraint of being a right angle triangle. Such triangles are called Heron triangles. In \cite{Goins}, Goins and Maddox have proved that the existence of a Heron triangle of area $n$ is equivalent to the existence of rational points of order greater than $2$ on an associated elliptic curve (given by (\ref{maddox})).  Elliptic curves associated with Heron triangles appear in the works of various authors. For a given integer $n$, the existence of infinitely many Heron triangles with the area $n$ has been studied in \cite{Rusin}. Buchholz and Rathbun \cite{Buchholz} proved the existence of infinitely many Heron triangles with two rational medians. Later in \cite{Buchholz2}, Buchholz and Stingley looked into the existence of Heron triangles with three rational medians. In \cite{dujella}, Dujella and Peral have shown the existence of elliptic curves of higher ranks associated with Heron triangles. In \cite{Halbeisen}, Halbeisen and Hungerbühler have shown the existence of elliptic curves of rank at least two associated with Heron triangles. In a recent work of Ghale et al. \cite{Ghale}, the authors constructed a family of elliptic curves of rank at most one from a certain Diophantine equation via Heron triangle.

\noindent In this article, we consider two primes $p$ and $q$ such that $p \equiv 7 \pmod 8$ and $q = 4^{m}+1$. We examine the group structure of elliptic curves associated with Heron triangles of area $2^{m}p$ having one of the angles as $\theta$ such that $\tau = \tan\frac{\theta}{2} = 2^{m}$. We prove the following theorem.

\begin{thm}\label{mainthm}
Let $p$ be a prime congruent to $7$ modulo $8$ and $q=4^m+1$ be a prime such that $\legendre{p}{q} = 1$. Then the $2$-Selmer rank of the elliptic curve  
\begin{equation}
    \label{ell1}
E : y^{2} = x(x-4^{m}p)(x+p)
\end{equation} is $1$ when $m=1$. In case $m \geq 2$, the $2$-Selmer rank of $E$ lies between $1$ and $3$.  
\end{thm}
\noindent Here, for an odd prime $l$ and an integer $a$ coprime to $l$, $\legendre{a}{l}$ denotes the Legendre symbol of $a$ modulo $l$. We deduce the following corollary under the assumption of the finiteness of the Shafarevich-Tate group. 
\begin{cor}
\label{pc}
 The Mordell-Weil rank of the the elliptic curve $E$ given by (\ref{ell1}) is at most $1$ when $m=1$. Moreover, if we assume the finiteness of the Shafarevich-Tate group $\Sha(E/\mathbb{Q})$, then the rank of $E(\mathbb{Q})$ is exactly $1$ and the $2$-part of $\Sha(E/\mathbb{Q})$ is trivial.   
\end{cor}
\noindent The corollary guarantees the existence of infinitely many Heron triangles with area $2p$ and one angle $\theta$ such that $\tau = \tan \frac{\theta}{2} = 2$, since for $q = 5$ we have infinitely primes $p \equiv 7 \pmod 8$ such that $\legendre{p}{5} = 1$ by Dirichlet's theorem on primes in arithmetic progression. 

\section{The $2$-Selmer Group}

\noindent 
We begin this section by recalling the association of a Heron triangle with an elliptic curve. In \cite{Goins}, Goins and Maddox have shown that any triangle $\Delta_{\tau,n}$ of area $n\in \mathbb{Z}$ with rational sides $a,b,c$ and an angle $\theta$ is associated with the elliptic curve 
\begin{equation}
\label{maddox}E_{\Delta_{\tau, n}}: \;\;y^{2} = x(x - n \tau) (x + n \tau^{-1}),
\end{equation}
where $\tau$ denotes $\tan\frac{\theta}{2}$. Moreover, they have shown that the torsion group of $E_{\Delta_{\tau, n}}(\mathbb{Q})$ will be either $\mathbb{Z}/ 2 \mathbb{Z} \times \mathbb{Z}/ 2 \mathbb{Z}$ or $\mathbb{Z}/ 2\mathbb{Z} \times \mathbb{Z}/ 4\mathbb{Z}$, the triangle being isosceles in the latter case. One can identify the elliptic curve given by (\ref{ell1})
with a Heron triangle of area $2^{m}p$ and an angle $\theta$ such that $\tau = \tan \frac{\theta}{2} = 2^{m}$. 

\noindent By the method of 2-descent (see \cite{Silverman}, Section $X.1.4$), there exists an injective homomorphism
$$ b: E(\mathbb{Q})/2E({\mathbb{Q}}) \longrightarrow \mathbb{Q}(S,2) \times \mathbb{Q}(S,2)$$
 defined by 
\begin{align*}
b(x,y) = \begin{cases}
(x, x - 4^{m}p)  & \text{if} \text{  } x\neq 0,4^{m}p, \\
(-1,-p) & \text{if} \text{  } x = 0,\\
(p,q) & \text{if}\text{  } x =4^{m}p, \\
(1,1) &\text{if} \text{  } \text{  } P = \mathcal{O},
\end{cases}
\end{align*}
where 
\begin{align}
\label{pairs}
\mathbb{Q}(S,2)&=\left\{b\in\mathbb{Q}^*/(\mathbb{Q}^*)^2 :
\text{ord}_l(b)\equiv 0 ~(\bmod \text{ } {2}) ~ \text{for all finite primes} ~ l\neq 2, p,q\right\}\\
&=\left\{\pm 1,\; \pm 2, \;\pm p, \;\pm q, \;\pm 2p, \;\pm 2q, \;\pm pq, \;\pm 2pq \right\}.\nonumber
\end{align}
Moreover, if $(b_1, b_2) \in \mathbb{Q}(S,2) \times \mathbb{Q}(S,2)$ is a pair that is not in the image of one of the three points $\mathcal{O} , (0,0), (4^{m}p,0)$, then $(b_1,b_2)$ is the image of a point $ P = (x,y) \in E(\mathbb{Q})/2E(\mathbb{Q})$ if and only if the equations
\begin{align}
& b_1z_1^2 - b_2z_2^2 = 4^{m}p, \label{eq22}\\
& b_1z_1^2 - b_1b_2z_3^2 = -p, \label{eq23}
\end{align}
have a solution $(z_1,z_2,z_3) \in \mathbb{Q}^* \times \mathbb{Q}^* \times \mathbb{Q}$. These equations represent a homogeneous space for $E$ (\cite{Silverman}) . 

\noindent The image of $E(\mathbb{Q})/2E({\mathbb{Q}})$ under the $2$-descent map is contained in a subgroup of $\mathbb{Q}(S,2)\times \mathbb{Q}(S,2)$ known as the $2$-Selmer group $\mbox{Sel}_2(E/\mathbb{Q})$, which fits into an exact sequence (see Chapter X, \cite{Silverman})
\begin{equation}
\label{selmer}
0 \longrightarrow E (\mathbb{Q})/2 E(\mathbb{Q}) \longrightarrow \mbox{Sel}_2 (E/\mathbb{Q}) \longrightarrow \Sha(E/\mathbb{Q})[2]\longrightarrow 0.
\end{equation}
The elements in $\mbox{Sel}_2(E/\mathbb{Q})$ correspond to the pairs $(a,b)\in \mathbb{Q}(S,2)^2$ such that the system of equations (\ref{eq22}) and (\ref{eq23}) have non-trivial local solutions in $\mathbb{Q}_{l}$ at all primes $l$ of $\mathbb{Q}$ including infinity. Note that $\# \,E(\mathbb{Q})/2 E(\mathbb{Q})=2^{2+r(E)}$. It is customary to denote $\#  \,\mbox{Sel}_2 (E/\mathbb{Q}) = 2^{2+s(E)}$, and refer to $s(E)$ as the $2$-Selmer rank. Clearly, we have
\begin{equation}
\label{ranks}
0 \leq r(E) \leq  s(E).
\end{equation}
From the inequality above as well as the exact sequence (\ref{selmer}), it can be easily seen that the $2$-Selmer rank controls both $E(\mathbb{Q})$ and $\Sha(E/\mathbb{Q})[2]$.  
\section{Local Solutions for the Homogeneous Spaces} 
\noindent In this section, we examine the properties of the $l$-adic solutions for the equations (\ref{eq22}) and (\ref{eq23}) that are associated with the $2$-Selmer group. In a later section, we use these properties to bound the size of the $2$-Selmer group. We use the well-known fact that any $l$-adic number $a$ can be written as $a = l^{n} \cdot u \text{ where } n \in \mathbb{Z}$, $u \in \mathbb{Z}_{l}^{*}$. We first prove the following result for all odd primes $l$.  
\begin{lem}\label{SHA-lem1}
Suppose the equations (\ref{eq22}) and (\ref{eq23}) have a solution $(z_1,z_2,z_3) \in \mathbb{Q}_{l} \times \mathbb{Q}_{l} \times \mathbb{Q}_{l}$ for any odd prime $l$. If $v_{l}(z_{i}) < 0$ for any one $i \in \{1,2\}$, then $v_{l}(z_{1})=v_l(z_2)=v_l(z_3)=-k < 0$ for some integer $k$. 
\end{lem}

\begin{proof}
Let $z_{i} = l^{k_{i}}u_{i}$, where $k_{i} \in \mathbb{Z}$ and $u_{i} \in \mathbb{Z}_{l}^{*}$ for $i = \{1,2,3\}$. Then $v_{l}(z_{i}) = k_{i}$ for all $i \in\{1,2,3\}$. 

\noindent Suppose $k_{1} < 0$. Then from equation (\ref{eq22}) one can get that $$b_{1}u_{1}^{2} - b_{2}u_{2}^{2}l^{2(k_{2}-k_{1})} = 4^{m}pl^{-2k_{1}}.$$ If $k_{2}>k_1$, then $l^{2}$ must divide $b_{1}$, a contradiction as $b_{1}$ is square-free. Hence $k_{2} \leq k_1<0$. Now if  $k_{2} < k_{1} < 0$ then again from equation (\ref{eq22}) we get $$b_{1}u_{1}^{2}l^{2(k_{1}-k_{2})} - b_{2}u_{2}^{2} = 4^{m}pl^{-2k_{2}},$$ which implies $l^{2}$ must divide $b_{2}$, a contradiction again. Hence if $k_{1} < 0$, then we have $k_{1} = k_{2} = -k < 0$ for some integer $k$. For $k_2<0$, one similarly gets $k_1=k_2=-k<0$.\\
From the equation (\ref{eq23}), we have $$b_{1}u_{1}^{2} - b_{1}b_{2}u_{3}^{2}l^{2(k_{3} - k_{1})} = -pl^{-2k_{1}}.$$ If $k_1<0$ and $k_3>k_1$, then $l^{2}$ must divide $b_{1}$, a contradiction as before. Hence $k_3\leq k_1<0$ if $k_1<0$. For $k_3 < k_{1} <0$, 
we can rewrite the above equation as 
\begin{equation}
\label{new1}
b_{1}u_{1}^{2}l^{2(k_{1} - k_{3})} - b_{1}b_{2}u_{3}^{2} = -pl^{-2k_{3}},
\end{equation} which implies $l^{2}$ must divide $b_{1}b_{2}$, i.e., $l = p$ or $q$ as $l$ is odd. If $l = p$, then from equation (\ref{new1}) we arrive at the contradiction that $p^{3}$ divides $b_{1}b_{2}$ whereas $b_1$ and $b_2$ are square-free. For $l=q$, we subtract the equation (\ref{eq23}) from the equation (\ref{eq22}) and observe that $$b_{1}b_{2}u_{3}^{2} - b_{2}u_{2}^{2}l^{2(k_{2} - k_{3})} = pql^{-2k_{3}}.$$
If $k_2>k_3$, we get a contradiction that $q^3$ divides $b_1b_2$. Therefore, $k_2 \leq k_3<0$ but then by the first part, $k_1 = k_2 \leq k_3$. Together, we obtain $k_1=k_2=k_3=-k<0$ for some integer $k$ if $k_1<0$ or $k_2<0$.
\end{proof}
\begin{lem}\label{SHA-lem2}
Suppose the equations (\ref{eq22}) and (\ref{eq23}) have a solution $(z_1,z_2,z_3) \in \big(\mathbb{Q}_{2}\big)^{3}$. If $b_1b_2 \equiv 2$ (mod $4$), then $v_{2}(z_{1})=v_2(z_2)=v_2(z_3)=-k < 0$. 
\end{lem}
\begin{proof} 
 Let $z_{i} = 2^{k_{i}}u_{i}$, where $k_{i} \in \mathbb{Z}$ and $u_{i} \in \mathbb{Z}_{2}^{*}$ for $i = \{1,2,3\}$. Then $v_{2}(z_{i}) = k_{i}$ for all $i = \{1,2,3\}$.\\ 
\noindent When $k_1<0$, the argument in the first part of the proof of Lemma \ref{SHA-lem1} also yields $k_1=k_2 =-k<0$ and $k_3\leq k_1$. From equation (\ref{new1}), we can also conclude that $k_1=k_3$ as $l^2\nmid b_1b_2$ for $l=2$ in this case.\\ 
\noindent If $k_1 > 0$, then $k_{3} \geq 0$ in equation (\ref{eq23}) implies $2$ divides $p$, a contradiction. If $k_{3} < 0< k_1$, then equation (\ref{eq23}) implies $b_{1}b_{2} \equiv 0 \pmod 4$, a contradiction again. \\
\noindent If $b_{1}$ is even, one can show that $k_{1} \neq 0$ by similar argument. If $b_{2}$ is even and $k_{1} = 0$, then $4$ divides $b_{1}$ by equation (\ref{eq22}), a contradiction if $k_{2} \geq 0$. But $k_{2} < 0$ contradicts the fact $k_{1} = 0$ by equation (\ref{eq22}). Hence the only possibility is $k_{1} = k_{2} = k_{3} = -k < 0$ for some integer $k$.
\end{proof}
\section{Bounding the Size of the $2$-Selmer Group}

\noindent In this section, we bound the size of the $2$-Selmer group of the Heronian elliptic curve given by (\ref{ell1}). The $2$-Selmer group $\mbox{Sel}_2 (E/\mathbb{Q})$ consists of those pairs $(b_1, b_2)$ in $\mathbb{Q}(S,2) \times \mathbb{Q}(S,2)$ for which the equation (\ref{eq22}) and (\ref{eq23}) have an $l$-adic solution at every place $l$ of $\mathbb{Q}$. 
We now limit the size of  $\mbox{Sel}_2 (E/\mathbb{Q})$ by ruling out local solutions for certain pairs $(b_{1},b_{2})$ by exploiting the results of the previous section.  
\begin{lem}\label{rank-lem2}
Let $(b_{1},b_{2}) \in \mathbb{Q}(S,2) \times \mathbb{Q}(S,2)$. Then\\
$(i)$ The corresponding homogeneous space will have no $l$-adic solution for the case $l = \infty$ if $b_{1}b_{2} < 0$.\\
$(ii)$ If $b_{1}b_{2}\equiv 2 \pmod 4$, the corresponding homogeneous space will not have $2$-adic solutions. \\
$(iii)$ If $b_{1} \equiv 0 \pmod q$, then the corresponding homogeneous space will not have any $q$-adic solution. 
\end{lem}
\begin{proof} $(i)$ Let the homogeneous space corresponding to $(b_{1},b_{2}) \in \mathbb{Q}(S,2) \times \mathbb{Q}(S,2)$ have real solutions. Then $b_{1} > 0$ and $b_{2} < 0$ implies $-p > 0$ in equations (\ref{eq23}), which is absurd. Similarly, $b_{1} < 0$ and $b_{2} > 0$ implies $4^{m}p < 0$ in equation (\ref{eq22}), which is absurd. Thus, the homogeneous space corresponding to $(b_{1},b_{2})$ have no $l$-adic solutions for $l = \infty$ if  $b_{1}b_{2}<0$. \\
\noindent $(ii)$ Let $b_{1}b_{2} \equiv 2 \pmod 4$. Suppose $b_{1}$ is even and $b_{2}$ is odd. Then $v_{2}(z_{1}) = v_{2}(z_{2}) = -k < 0$ for some integer $k$ from lemma (\ref{SHA-lem2}). This in turn implies 
\begin{equation}
    b_{1}u_{1}^{2} - b_{2}u_{2}^{2} = 4^{m}p\cdot 2^{2k} \implies b_{2} \equiv 0 \pmod 2,
\end{equation} a contradiction. The case when $b_1$ is odd follows similarly. \\
\noindent $(iii)$ Let us assume $b_{1} \equiv 0 \pmod q$. Then one of $v_{q}(z_{1})$ or $v_{q}(z_{3})$ has to be negative in equation (\ref{eq23}). If $v_{q}(z_{1}) < 0$, then from lemma (\ref{SHA-lem1}), we have  $v_{q}(z_{1}) = v_{q}(z_{2}) = v_{q}(z_{3}) = -k < 0$ for some integer $k$. Subtracting equation (\ref{eq23}) from the equation (\ref{eq22}), we get
\begin{equation}
    b_{1}b_{2}u_{3}^{2} - b_{2}u_{3}^{2} = pq^{2k+1},
\end{equation} where $z_{i} = u_{i}q^{i}$ for $u_{i} \in \mathbb{Z}_{q}^{*}$ and $i \in {1,2,3}$. This implies that $b_{2} \equiv 0 \pmod q$. From  equation (\ref{eq23}), one can now deduce that
\begin{equation}
    b_{1}u_{1}^{2} - b_{1}b_{2}u_{3}^{2} = -pq^{2k} \implies b_{1} \equiv 0 \pmod {q^2},
\end{equation} a contradiction. For the case $v_{q}(z_{3}) < 0$, if additionally one of $v_{q}(z_{i}) < 0$ for $i \in \{1,2\}$, we are back to the previous case due to lemma (\ref{SHA-lem1}).\\
\noindent So we now suppose $v_{q}(z_{3}) < 0$ and $v_{q}(z_{i}) \geq 0$ for $i \in \{1,2\}$. Then from equation (\ref{eq23}), one can immediately observe $b_{1}b_{2} \equiv 0 \pmod {q^{2}} \implies b_{2} \equiv 0 \pmod q$ too. From equation (\ref{eq22}), this in turn implies that $4^{m}p \equiv 0 \pmod q$, a contradiction again. Hence the result follows.
\end{proof}
\noindent With the help of the lemma (\ref{rank-lem2}) and noting that the image of the torsion points $(0,0), (4^{m},0), (-p,0)$ and $\mathcal{O}$ under the map $b$ is $$A = \{(-1,-p), (p,q), (-p,-pq), (1,1)\},$$ we can now solely focus on the seven pairs 
\begin{equation}
    \label{seven}
(1,p), \;(1,q),\; (1,pq), \;(2,2), \;(2,2p),\; (2,2q),\; (2,2pq).
\end{equation}

\noindent Every other pair $(b_{1},b_{2})$ will either belong to the same coset of one of those seven points in the torsion group $Im(b)/A$ or the corresponding homogeneous space will not yield local solutions for at least one prime $l \leq \infty$ by the lemma (\ref{rank-lem2}) and consequently will not have rational solutions either. \noindent The following result narrows down the possible pairs from seven to three.

\begin{lem}\label{Sha-lem4}
There are no $p$-adic solutions to the homogeneous spaces corresponding to $(1,p)$, $(1,pq)$, $(2,2p)$ and $(2,2pq)$. 
\end{lem}
\begin{proof} First we prove the result for the case $(b_{1},b_{2}) = (1,p)$. A very similar proof can also be carried out for the case $(b_{1},b_{2}) = (1,pq)$.\\
\noindent If $v_{p}(z_{1}) > 0$ then assuming $z_{1} = p \Tilde{z}_{1}$ one can get the following from equation (\ref{eq22}). 
\begin{equation}
    p^{2} \Tilde{z}_{1}^{2} - pz_{2}^{2} = 4^{m}p \implies -z_{2}^{2} \equiv 4^{m} \pmod p
\end{equation} which implies that $\big(\frac{-1}{p}\big) = 1$, a contradiction as $p \equiv 7 \pmod 8$.\\
Now $v_p(z_1)=0$ implies $v_{p}(z_{2}) < 0$ from equation (\ref{eq22}), which in turn implies $v_{p}(z_{1}) < 0$ from lemma (\ref{SHA-lem1}), a contradiction. \\
If $v_p(z_1)<0$ then from lemma (\ref{SHA-lem1}), $v_{p}(z_{1}) = v_{p}(z_{2}) = -k < 0$ for some integer $k$. Assuming $z_{i} = u_{i}p^{-k}$ for $i \in \{1,2\}$, equation (\ref{eq22}) yields $u_{1}^{2} - pu_{2}^{2} = 4^{m}p^{2k+1} \implies u_{1} \equiv 0 \pmod p$, a  contradiction. 

\noindent We now deal with the pairs $(2,2p)$, the argument being similar for $(2,2pq)$. From the equations (\ref{eq22}) one can see that  if $v_{p}(z_{2}) \geq 0 $ then $v_{p}(z_{1}) > 0$. Now $v_{p}(z_{1}) > 0$ implies $\big(\frac{-2}{p}\big) = 1$, a contradiction as $p \equiv 7 \pmod 8$. If $v_{p}(z_{2}) < 0$, then $v_{p}(z_{2}) = v_{p}(z_{1}) = -k < 0$ for some integer $k$. Assuming $z_{i} = u_{i}p^{k}$ where $u_{i} \in \mathbb{Z}_{p}^{*}$ for $i \in \{1,2,3\} $, one can now observe from equation $(\ref{eq22})$ that
\begin{equation}
    2u_{1}^{2} - 2pu_{2}^{2} = 4^{m}p^{2k+1} \implies u_{1} \equiv 0 \pmod p,
\end{equation} a contradiction again. Hence the result follows. 
\end{proof}
\noindent We can reduce the possibilities for $(b_1,b_2)$ further down from three to one if $2$ is not a quadratic residue modulo $q$, i.e., $q=5$. 
\begin{lem}\label{Sha-lem3}
Suppose $m=1$, i.e., $q=5$. Then the homogeneous spaces corresponding to $(b_{1},b_{2}) \in \{(2,2), (2,2q)\}$ will not have $q$-adic solution.  
\end{lem}

\noindent \textbf{Proof:} 
First consider $(b_{1},b_{2}) = (2,2)$. Subtracting equation (\ref{eq23}) from equation (\ref{eq22}) we get
\begin{equation}
    4z_{3}^{2} - 2z_{2}^{2} = 5p.
\end{equation}
This implies either $z_{2} \equiv z_{3} \equiv 0 \pmod 5$ or $\big(\frac{2}{5}\big) = 1$, the latter being an immediate contradiction. If $z_{2} \equiv z_{3} \equiv 0 \pmod 5$, equation (\ref{eq23}) implies $2z_{1}^{2} \equiv -p \pmod 5$, which is a contradiction again as $\big(\frac{-p}{5}\big) = 1$ but $\big(\frac{2}{5}\big) = -1$. Hence the result follows for $(b_{1},b_{2}) = (2,2)$. For the case $(b_{1},b_{2}) = (2,2q)$ the result follows from equation (\ref{eq22})
\begin{equation}
    2z_{1}^{2} - 10 z_{2}^{2} = 4p \implies 2z_{1}^{2} \equiv 4p \pmod 5.
\end{equation}
This in turn implies $\big(\frac{2}{5}\big) = 1$, a contradiction. $\hfill \square$

\begin{lem}\label{SHA-lem5}
The equations (\ref{eq22}) and (\ref{eq23}) have a local solution in $\mathbb{Q}_{l}$ for every prime $l$ for $(b_1,b_2)=(1,q)$.
\end{lem}
\begin{proof}
First we consider $l\geq 5$. Suppose $C$ is the homogeneous space given by the equations (\ref{eq22}) and (\ref{eq23}) corresponding to the pair $(1,q)$. Then $C$ is a twist of $E$, and in particular, it has genus $1$. By the Hasse-Weil bound, we have 
$$\# C(\mathbb{F}_{l})\;\geq \;1+l-2\sqrt{l}\;\geq \;2 \qquad \mbox{ for } l\geq 5.$$
We can choose a solution $(z_{1},z_{2},z_{3}) \in \mathbb{F}_{l} \times \mathbb{F}_{l} \times \mathbb{F}_{l}$ such that not all three of them are zero modulo $l$. Now $z_{1} \equiv z_{2} \equiv 0 \pmod l$ implies $l^{2}$ divides $4^{m}p$, a contradiction. Similarly, $z_{1} \equiv z_{3} \equiv 0 \pmod l$ implies $-p \equiv 0 \pmod l^{2}$, contradiction again. Hence for $l \geq 5$, there exists solution to the simultaneous equations (\ref{eq22}) and (\ref{eq23}) in $\mathbb{F}_{l}$ that can be lifted to $\mathbb{Q}_{l}$ via Hensel's lemma. 

For $l=3$, first let us observe that $q \equiv 2 \pmod 3$ always. If $p \equiv 1 \pmod 3$, choose $z_{2} = 0$ and $z_{3} = 1$. Then $z_{1}^{2} = 4^{m}p \equiv 1 \pmod 3$ from first equation and $z_{1}^{2} = q-p \equiv 1 \pmod 3$ from second equation. Hence $z_{1} \not \equiv 0 \pmod 3$ is a solution that can be lifted by Hensel's lemma. If $p \equiv 2 \pmod 3$, choose $z_{2} = 1, z_{3} = 0$. Then $z_{1}^{2} = q + 4^{m}p \equiv 1 \pmod 3$ from first equation and $z_{1}^{2} = -p \equiv 1 \pmod 3$ from second equation. Hence $z_{1} \not \equiv 0 \pmod 3$ is a solution that can be lifted by Hensel's lemma.

Finally for the case $l=2$, choose $z_{2} = 1$ and $z_{3} = 0$. For $m=1$, this turns equation (\ref{eq22}) and equation (\ref{eq23}) into the following;
\begin{align}
    z_{1}^{2} - 5z_{2}^{2} = 4p \equiv 28 \equiv 4 \pmod 8,\\
    z_{1}^{2} - 5z_{3}^{2} = -p \equiv 1 \pmod 8.
\end{align} In both the cases we now have $z_{1}^{2} \equiv 1 \pmod 8$ which by Hensel's lemma can be lifted to $\mathbb{Q}_{2}$. Similarly for $m \geq 2$, one can immediately observe that $z_{1}^{2} \equiv 1 \pmod 8$ again, and hence can be lifted similarly to $\mathbb{Q}_{2}$ via Hensel's lemma. Hence proved.   
\end{proof}
\noindent By Lemmas \ref{rank-lem2}, \ref{Sha-lem4} and \ref{SHA-lem5}, we can conclude that the the $2$-Selmer rank of $E$ lies between $1$ and $3$ as stated in Theorem \ref{mainthm}. For $q=5$, it follows from Lemma \ref{Sha-lem3} that the $2$-Selmer rank is exactly $1$, concluding the proof of Theorem \ref{mainthm}. 

\section{The Mordell-Weil Rank and $\Sha(E/\mathbb{Q})[2]$}

\noindent We have seen that the $2$-Selmer group has rank $1$ in the previous section when $q=5$. By the exact sequence (\ref{selmer}), it follows that either $E(\mathbb{Q})$ has rank $0$ or $\Sha(E/\mathbb{Q})[2]=0$. If we assume the finiteness of $\Sha(E/\mathbb{Q})$ as predicted by Shafarevich, then $\Sha(E/\mathbb{Q})$ must have square order by Cassels-Tate pairing. Therefore, $\Sha(E/\mathbb{Q})[2]$ has to be $0$ i.e. the Mordell-Weil rank of $E$ is $1$ as stated in Corollary \ref{pc}. This also proves the existence of infinitely many Heron triangles with the previously mentioned properties. 

\noindent We include a list of Heronian elliptic curves satisfying the properties mentioned in the Theorem (\ref{mainthm}) with the corresponding rank, $2$-Selmer rank and Shafarevich-Tate group in the below table. The computations have been done in Magma and SageMath software. There is a disparity between the definition of the $2$-Selmer rank defined in this work and that in Magma and SageMath. While we define $\#  \,\mbox{Sel}_2 (E/\mathbb{Q}) = 2^{2+s(E)}$, and refer to $s(E)$ as the $2$-Selmer rank in the following table, $2+s(E)$ will denote the $2$-Selmer rank in Magma or SageMath. \\
\begin{center}
\begin{tabular}{ |c|c|c|c|c|c|c|c|c|c| }
\hline

 $p$ & $q$ &  $r(E)$ & $s(E)$ &  $\Sha(E/\mathbb{Q})[2]$ & $p$ & $q$ & $r(E)$ & $s(E)$ &  $\Sha(E/\mathbb{Q})[2]$\\ 
 \hline
31 & 5 & 1 & 1 & trivial & 47 & 17 & 0 & 2 & no information\\
 \hline
71 & 5 & 1 & 1 & trivial & 127 & 17 & 2 & 2 & no information\\
 \hline
151 & 5 & 1 & 1 & trivial & 191 & 17 & 0 & 2 & no information\\
\hline
\end{tabular}
\end{center}
\vspace{1cm}
\noindent {\it Acknowledgement:} The first author would like to acknowledge DST-SERB for providing the grant through the Start-Up Research Grant (SRG/2020/001937) as well as BITS-Pilani, Hyderabad Campus for providing amenities. The second author would like to acknowledge the fellowship (File No:09/1026(0029)/2019-EMR-I) and amenities provided by the Council of Scientific and Industrial Research, India (CSIR) and BITS-Pilani, Hyderabad. The third author would like to thank MATRICS, SERB for their research grant MTR/2020/000467.


\begin{thebibliography}{99}

\bibitem{Birch} B.J.Birch, H.P.F. Swinnerton-Dyer. "Notes on elliptic curves. II." Journal für die reine und angewandte Mathematik, 218 (1965), 79-108.

\bibitem{Brown1} Heath-Brown D.R. "The size of Selmer groups for the congruent number problem." Inventiones Mathematicae 111.1 (1993): 171-195.

\bibitem{Brown2} Heath-Brown, D. R. "The size of Selmer groups for the congruent number problem, II." Inventiones Mathematicae 118 (1994): 331-370.

\bibitem{Buchholz} Buchholz, Ralph H., and Randall L. Rathbun. "An infinite set of Heron triangles with two rational medians." The American Mathematical Monthly 104.2 (1997): 107-115.

\bibitem{Buchholz2} Buchholz, Ralph H., and Robert P. Stingley. "Heron triangles with three rational medians." The Rocky Mountain Journal of Mathematics 49.2 (2019): 405-417.

\bibitem{Chahal} Chahal, Jasbir S. "Congruent numbers and elliptic curves." The American Mathematical Monthly 113.4 (2006): 308-317.
\bibitem{Coates} Coates, John H. "Congruent number problem." Pure and Applied Mathematics Quarterly 1.1 (2005): 14-27.

\bibitem{Connel}Connell, Ian. "Calculating root numbers of elliptic curves over Q." Manuscripta Mathematica 82.1 (1994): 93-104.

\bibitem{conrad} Conrad, Keith. "The congruent number problem." The Harvard College Mathematics Review 2.2 (2008): 58-74.



\bibitem{dujella} Dujella, Andrej, and Juan Carlos Peral. "Elliptic curves coming from Heron triangles." The Rocky Mountain Journal of Mathematics 44.4 (2014): 1145-1160.

\bibitem{Ghale} Ghale, Vinodkumar, Shamik Das, and Debopam Chakraborty. "On the Diophantine equation $(x^{2}+ y^{2})^{2}+(2pxy)^{2}= z^{2} $." arXiv e-prints (2021): arXiv-2105.

\bibitem{Goins} Goins, Edray Herber, and Davin Maddox. "Heron triangles via elliptic curves." The Rocky Mountain Journal of Mathematics (2006): 1511-1526.

\bibitem{Halbeisen} Halbeisen, Lorenz, and Norbert Hungerbühler. "Heron triangles and their elliptic curves." Journal of Number Theory 213 (2020): 232-253.


\bibitem{Johnstone} Johnstone, Jennifer A., and Blair K. Spearman. "Congruent number elliptic curves with rank at least three." Canadian Mathematical Bulletin 53.4 (2010): 661-666.

\bibitem{Kramer} Kramer, A. V., and F. Luca. "Some remarks on Heron triangles." Acta Acad. Paed. Agriensis, Sectio Mathematicae 27 (2000) 25–38.





\bibitem{Rhodes} Rhoades, Robert C. "2-Selmer groups and the Birch–Swinnerton-Dyer Conjecture for the congruent number curves." Journal of Number Theory 129.6 (2009): 1379-1391.

\bibitem{Rusin} Rusin, David J. "Rational triangles with equal area." New York J. Math 4.1 (1998): 15..


\bibitem{Sierpinski} Sierpinski, Waclaw. Pythagorean triangles. Vol. 9. Courier Corporation, 2003.

\bibitem{Silverman} Silverman, Joseph H. The arithmetic of elliptic curves. Vol. 106. New York: Springer, 2009.


\bibitem{Yan} Yan, Xiao-Hui. "The Diophantine equation $(m^2+ n^2)^x+(2mn)^y=(m+ n)^{2z}$." International Journal of Number Theory 16.08 (2020): 1701-1708.



\end{thebibliography}
\end{document}